\documentclass[preprint,authoryear]{elsarticle}
\usepackage{graphicx}
\usepackage{amsmath}
\usepackage{amsfonts}
\usepackage{amssymb,bm}
\usepackage{amsthm}
\usepackage[utf8]{inputenc}
\usepackage{booktabs}
\usepackage{caption}
\usepackage{subcaption}
\usepackage{hyperref}
\usepackage[symbol]{footmisc}
\usepackage{tikz}
\usetikzlibrary{trees}
\usepackage{dsfont}
\usepackage{algorithm}
\usepackage{algpseudocode}

\newtheorem{theorem}{Theorem}

\newtheorem{definition}[theorem]{Definition}

\newtheorem{lemma}[theorem]{Lemma}

\newtheorem{proposition}[theorem]{Proposition}
\newtheorem{remark}[theorem]{Remark}

\newcommand{\fin}{n}

\newcommand{\prob}{\mathbb P}

\newcommand{\bX}{\mathbf{X}}

\newcommand{\nei}{\text{ne}}
\renewcommand{\P}{\mathbb P}

\newcommand{\Perp}{\!\perp \! \! \! \perp\!}

\begin{document}

\title{Model selection for Markov random fields on graphs under a mixing condition\tnoteref{t1}}
\tnotetext[t1]{
This work was produced as part of the activities of the \emph{Research, Innovation and Dissemination Center for Neuromathematics} (grant FAPESP 2013/07699-0). It was also supported by FAPESP project (grant 2017/10555-0) \emph{``Stochastic Modeling of Interacting Systems''}  and CNPq Universal project (grant 432310/2018-5) \emph{``Statistics, stochastic processes and discrete structures''}. }

\author{Florencia Leonardi\corref{cor1}\fnref{fn1}}
\ead{florencia@usp.br}
\affiliation{organization={Instituto de Matemática e Estatística, Universidade de São Paulo},
addressline={Rua do Matão 1010},
postcode={05508-090},
city={São Paulo},
country={Brazil}}

\author{Magno T. F. Severino\fnref{fn2}}
\ead{magnotfs@usp.br}

\cortext[cor1]{Corresponding author}
\fntext[fn1]{Partially supported by a CNPq’s research fellowship, grant 311763/2020-0.}
\fntext[fn2]{Supported by CAPES and CNPq  PhD fellowships.}

\begin{abstract}
In this work, we propose a global model selection criterion to estimate the
graph of conditional dependencies of a random vector based on a finite sample. 
By global criterion, we mean optimizing a function over the entire set of possible graphs, eliminating the need to estimate the individual neighborhoods and subsequently combine them to estimate the graph. We prove the almost sure convergence of the graph estimator.
This convergence holds provided the data is a realization of a multivariate stochastic process that satisfies a mixing condition. 
To the best of our knowledge, these are the first results to show the consistency of a model selection criterion for Markov random fields on graphs under non-independent data. 
\end{abstract}

	
\maketitle

\section{Introduction}

We consider a vector-valued stochastic process, denoted as $X^{(1)},\dotsc, X^{(n)}$, with values in $A^d$, where $A$ is a finite alphabet. We assume the process is stationary, with invariant distribution $\pi$. The stationarity condition is not necessary to derive the results we state in the paper, but it is a convenient assumption from the notational point of view. Denote by $G^*$ the graph encoding the conditional dependencies in  $\pi$. Our primary goal in this work is to estimate $G^*$ and the associated conditional probability distributions.

In the case where we assume that the sample $X^{(1)},\dotsc, X^{(n)}$ is independent and identically distributed, we reduce to the classical model selection for discrete graphical models or Markov random fields on graphs. Extensive research has been conducted on these models, including, but not limited to,  \citep{lauritzen1996, koller2009,lerasle_et_al2016, pensar2017,divino2000,leonardi_et_al2023}. 
Furthermore, these models have found applications in various fields, including Biology \cite{shojaie2010}, Social Sciences \cite{strauss1990} or Neuroscience \cite{duarte_et_al2019}. 
Up to this moment, the most studied model has been the binary graphical model with pairwise interactions where structure estimation can be  addressed by using standard logistic regression techniques \citep{strauss1990, ravikumar2010}, distance-based approaches between conditional probabilities  \citep{galves2015, bresler_et_al2018} and maximization of the $\ell_1$-penalized pseudo-likelihood \citep{atchade2014, hoefling09a}; see also \citet{narayana2012}.  
In the case of bigger discrete alphabets or general types of interactions, 
to our knowledge, the only works addressing the structure estimation problem are \cite{loh2013,leonardi_et_al2023}. In \cite{loh2013}, the authors obtain a characterization of the edges in the graph with the zeros in a generalized inverse covariance matrix. Then, this characterisation is used to derive estimators for restricted classes of models, and the authors prove the consistency in probability of these estimators. In the work \citep{leonardi_et_al2023}, a penalized criterion is proposed to estimate the neighborhood of each vertex, and the results are combined to construct the model's graph. Markov random fields on graphs have also been proposed for continuous random variables, where the structure estimation problem has been addressed by $\ell_1$-regularization for Gaussian Markov random fields \citep{meinshausen2006} and also extended to non-parametric 
 models \citep{lafferty2012, liu2012} and  general conditional distributions from the exponential family \cite{yang2015}.

From another perspective, graphical models can be seen as non-homogeneous versions of general random fields or Gibbs distributions on lattices, classical models in stochastic processes, and statistical mechanics theory \cite{georgii2011}.
In such a setting, despite having only one observation within the sample, the number of variables increases.
Given the regularity of the graph (each node has the same neighborhood), inference and model selection can be done based on the unique observation. 
The statistical inference for Markov random fields and Gibbs distributions under this setting has been addressed in Francis Comets' works \citep{comets1992,comets-gidas1992}. 
More recently, model selection criteria, such as the BIC proposed by \citet{schwarz1978}, have been proven consistent under this regular setting \citep{ji1996,csiszar2006b}; see also \citet{tjelmeland1998} and \citet{locherbach2011}. 

From an applied point of view, the assumption of independence of the observations in the non-homogeneous Markov random fields setting is often too restrictive.
Consider, for example, the task of estimating interaction graphs from EEG time series data \citep{cerqueira2017}, river stream flow data \citep{leonardi2020} or daily stock market indices \citep{leonardi_et_al2023}. 
In these scenarios, the independence assumption does not hold, and the methods commonly used for graphical models serve only as approximations to the true underlying distribution. 
While such approximations can be practical from an applied point of view, 
from a theoretical perspective it is interesting to consider the problem of estimation and model selection in a dependence scenario, as for example the case of mixing processes considered here. 

Conventional model selection techniques for graphical models often involve estimating the neighborhoods of individual nodes and constructing the graph based on these neighborhoods, as exemplified by \citet{ravikumar2010}.
But depending on the rule to combine the neighborhoods, the final estimated graph can drastically underestimate or overestimate the set of edges in the graph \citep{leonardi_et_al2023}.
In this work, we adopt a global estimation perspective that overcomes this limitation.
Our approach involves estimating the graph by optimizing the penalized pseudo-likelihood function over the set of all possible simple and undirected graphs.
We provide a proof of convergence, showing that the estimator almost surely converges to the true underlying graph in cases of finite graphical models, provided a mixing condition holds for the generating process.

The paper is organized as follows. 
In Section~2, we provide essential definitions and notations concerning classical graphical models. 
Section~3 is dedicated to introducing the vector-valued mixing process and presenting important auxiliary theoretical results. 
In Section~4, we introduce the penalized maximum pseudo-likelihood estimator of the graph of conditional dependencies and state and prove the main consistency result of the paper.

\section{Markov random fields on graphs}

A \emph{graph} is defined as an ordered pair $G = (V, E)$, where $V$ represents the set of vertices (or nodes), and $E \subseteq V \times V$ is the set of edges connecting pairs of vertices.
We refer to a graph as \emph{undirected} if $(v_i, v_j) \in E$ implies that $(v_j, v_i) \in E$ for all $(v_i, v_j) \in E$, where $v_i, v_j \in V$. 
Furthermore, a graph is considered \emph{simple} if $(v, v) \notin E$ for all $v \in V$.
For the purposes of this work, we concentrate exclusively on undirected simple graphs, which we will henceforth simply call a \emph{graph}.

Consider a graph $G=(V,E)$, with $V = \{1, \ldots, d\}$, for $d \in \mathbb{N}$, and assume we observe at each vertex $v\in V$ a random variable $X_v$, which is discrete and takes values in $A$, a finite alphabet. 
Moreover, let $X = (X_1, \ldots, X_d)$ be the vector of all variables observed on the vertices of the graph. Denote by $\P$ the joint  probability distribution of the vector $X$.  For any $W \subset V$ and any configuration $a_w \in A^{|W|}$ we write
\begin{equation*}
    \pi(a_W) = \mathbb{P}(X_W = a_w)\,.
\end{equation*}
Moreover, if $\pi(a_W) > 0$ then we denote by
\begin{equation*}
    \pi(a_U \vert a_W) = \mathbb{P}(X_U = a_U \vert X_W = a_w),
\end{equation*}
for $a_U \in A^{|U|}$ and $a_W \in A^{|W|},$ the corresponding conditional probability distributions.

For a given vertex $v \in V$, any set $W \subset V$, with $v \notin W$, is called a \emph{neighborhood} of $v.$
Furthermore, $W$ is called \emph{Markov neighborhood} of $v$ if
\begin{equation*}
    \pi(a_v \vert a_U) = \pi(a_v | a_W)
\end{equation*}
for all $U \supset W$, 
$v \notin U$ and 
all $a_U \in A^{|U|}$, 
with $\pi(a_U)>0$. 
The definition of a Markov neighborhood $W$ of $v$ is equivalent to request that for all $U^\prime \subset V \setminus \{v\}$ such that $U^\prime\cap W = \emptyset$, $X_{U^\prime}$ is conditionally independent of $X_v$, given $X_W$. That is, 
\begin{equation*}
X_v \Perp X_{U^\prime} \vert X_W
\end{equation*}
for all $U^\prime$ with $U^\prime \cap W = \emptyset$, where $\Perp$ is the usual symbol denoting independence of random variables.

\begin{figure}[t!]
\begin{center}
\begin{minipage}{6cm}
\begin{center}
\includegraphics*[scale=0.85]{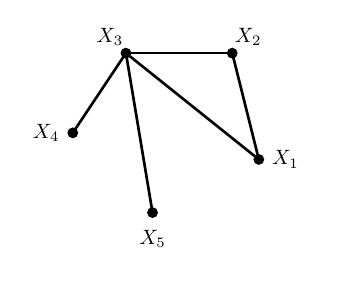}\\
(a)
\end{center}
\end{minipage}
\begin{minipage}{6cm}
\begin{center}
\includegraphics*[scale=0.56]{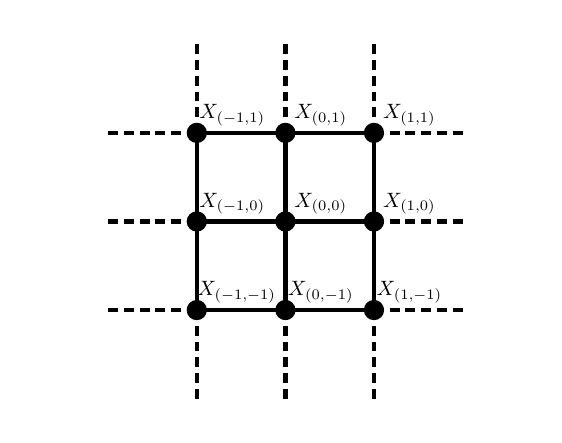}\\
(b)\\[3mm]
\end{center}
\end{minipage}
\end{center}
\caption{Different graph structures for Markov random fields under different settings.
The figure on the left is a finite graphical model or Markov random field on a general graph, and the figure on the right represents the interaction graph in a classical Markov random field or Gibbs distribution on a regular lattice.}
\label{graphs}
\end{figure}

As discussed in \cite{leonardi_et_al2023}, if $W$ is a Markov neighborhood of $v\in V$, then any finite set $U\supset W$ also is a Markov neighborhood of $v$.
In contrast, $W_1$ and $W_2$ being Markov neighborhoods of $v$ does not imply in general that $W_1\cap W_2$ is a Markov neighborhood of $v,$ however this property is satisfied by some probability measures. 
This fact leads to the following definition.

\begin{definition}[Markov intersection property] \label{def:MIP}
For all $v\in V$ and all $W_1$ and $W_2$ Markov neighborhoods of $v$, the set $W_1\cap W_2$ is also a Markov neighborhood of $v$.
\end{definition}

The Markov intersection property is desirable in this context to define the smallest Markov neighborhood of a node. This property is guaranteed under the usually assumed positivity condition \citep[see][]{lauritzen1996}.  But the positivity assumption is not necessary to obtain consistent estimators, and then it is enough to assume the Markov intersection property, see  \citet{leonardi_et_al2023} for details. 

\begin{definition}[Basic neighborhood]\label{def:ne_v}
For $v \in V$, let $\mathcal{W}(v)$ be the set of all subsets of $V$ that are Markov neighborhoods of $v$.
The \emph{basic neighborhood} of $v$ is defined as
\begin{equation}\label{eq:ne_v}
    \nei(v) = \bigcap_{W \in \mathcal{W}(v)} W.
\end{equation}
\end{definition}

By the  Markov intersection property, $\nei(v)$ is the smallest Markov neighborhood of $v \in V$.
Based on these basic neighborhoods, define the graph $G^*=(V,E^*)$ as 
\begin{equation} \label{eq:graph}
(v,w) \in E^* \; \text{ if and only if } w \in \nei(v),
\end{equation}
where $E^* \subseteq V \times V.$
The graph $G^*$ with edges defined in \eqref{eq:graph} is \emph{undirected}, as proved by \citep{leonardi_et_al2023}.
Figure~\ref{graphs} shows two examples of graphs for Markov random fields under different settings: the finite non-homogeneous graphical model case in (a) and the interaction graph in a classical Markov random field, or Gibbs distribution, on a regular lattice \citep{comets-gidas1992,csiszar2006b} in (b).

\section{Vector-valued mixing processes}

In this paper we consider a vector-valued stochastic process $X^{(1)},X^{(2)},$ $\dotsc$, where each variable $X^{(i)}$ is a vector of $d$ components, belonging to the set $A^d$, with $A$ a finite alphabet. 
We denote by $(A^{d\times\infty}, \mathcal F, \P)$ the probability space for the process $\{X^{(i)}\colon i\in \mathbb N\}$. 
Sometimes we need to consider ``slices'' of the entire realization $X^{(1)},\dotsc, X^{(n)}$ on both dimensions. 
To avoid misleading notations we use superscripts to denote the indexes in ``time'' (ranging from 1 to $n$) and subscripts to denote indexes on ``space'' (a subset of $V=\{1,\dots,d\}$). 
For any set $U\subset V$ and any integer interval $i:j$ we denote by $X_U^{(i:j)}$ the sequence $X^{(i)}_U,\dots, X^{(j)}_U$ with $X_U^{(k)}=(X_u^{(k)}\colon u\in U)$, $k=i,\dots,j$. When $U=V$ we avoid the subscript and simply write  $X^{(i:j)}$. The same notation is used for ``realizations'' of the process, denoted in lower case $x_U^{(i:j)}$ instead of the notation for the random variables  $X_U^{(i:j)}$. See an example in Figure~\ref{fig:exemplo_processo}.

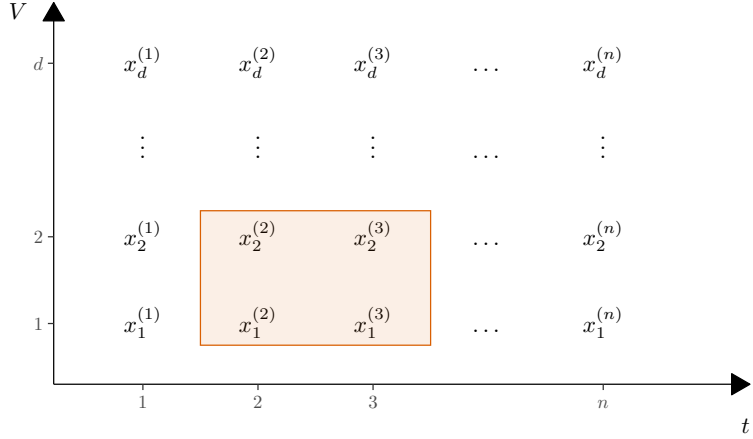
\begin{figure}[t]
    \centering
    \scalebox{0.8}{
\begin{tikzpicture}[x=1pt,y=1pt]
\definecolor{fillColor}{RGB}{255,255,255}
\path[use as bounding box,fill=fillColor,fill opacity=0.00] (0,0) rectangle (361.35,216.81);
\begin{scope}
\path[clip] (  0.00,  0.00) rectangle (361.35,216.81);
\definecolor{drawColor}{RGB}{255,255,255}
\definecolor{fillColor}{RGB}{255,255,255}

\path[draw=drawColor,line width= 0.6pt,line join=round,line cap=round,fill=fillColor] (  0.00,  0.00) rectangle (361.35,216.81);
\end{scope}
\begin{scope}
\path[clip] ( 26.64, 30.69) rectangle (355.85,211.31);
\definecolor{fillColor}{RGB}{255,255,255}

\path[fill=fillColor] ( 26.64, 30.69) rectangle (355.85,211.31);
\definecolor{drawColor}{RGB}{0,0,0}

\node[text=drawColor,anchor=base,inner sep=0pt, outer sep=0pt, scale=  1.10] at ( 68.81, 55.62) {$x_1^{(1)}$};

\node[text=drawColor,anchor=base,inner sep=0pt, outer sep=0pt, scale=  1.10] at ( 68.81, 96.67) {$x_2^{(1)}$};

\node[text=drawColor,anchor=base,inner sep=0pt, outer sep=0pt, scale=  1.10] at ( 68.81,178.77) {$x_d^{(1)}$};

\node[text=drawColor,anchor=base,inner sep=0pt, outer sep=0pt, scale=  1.10] at (123.23, 55.62) {$x_1^{(2)}$};

\node[text=drawColor,anchor=base,inner sep=0pt, outer sep=0pt, scale=  1.10] at (123.23, 96.67) {$x_2^{(2)}$};

\node[text=drawColor,anchor=base,inner sep=0pt, outer sep=0pt, scale=  1.10] at (123.23,178.77) {$x_d^{(2)}$};

\node[text=drawColor,anchor=base,inner sep=0pt, outer sep=0pt, scale=  1.10] at (177.64, 55.62) {$x_1^{(3)}$};

\node[text=drawColor,anchor=base,inner sep=0pt, outer sep=0pt, scale=  1.10] at (177.64, 96.67) {$x_2^{(3)}$};

\node[text=drawColor,anchor=base,inner sep=0pt, outer sep=0pt, scale=  1.10] at (177.64,178.77) {$x_d^{(3)}$};

\node[text=drawColor,anchor=base,inner sep=0pt, outer sep=0pt, scale=  1.10] at (286.47, 55.62) {$x_1^{(n)}$};

\node[text=drawColor,anchor=base,inner sep=0pt, outer sep=0pt, scale=  1.10] at (286.47, 96.67) {$x_2^{(n)}$};

\node[text=drawColor,anchor=base,inner sep=0pt, outer sep=0pt, scale=  1.10] at (286.47,178.77) {$x_d^{(n)}$};

\node[text=drawColor,anchor=base,inner sep=0pt, outer sep=0pt, scale=  1.10] at (232.06, 55.62) {$\dots$};

\node[text=drawColor,anchor=base,inner sep=0pt, outer sep=0pt, scale=  1.10] at (232.06, 96.67) {$\dots$};

\node[text=drawColor,anchor=base,inner sep=0pt, outer sep=0pt, scale=  1.10] at (232.06,137.72) {$\dots$};

\node[text=drawColor,anchor=base,inner sep=0pt, outer sep=0pt, scale=  1.10] at (232.06,178.77) {$\dots$};

\node[text=drawColor,anchor=base,inner sep=0pt, outer sep=0pt, scale=  1.10] at ( 68.81,137.72) {$\vdots$};

\node[text=drawColor,anchor=base,inner sep=0pt, outer sep=0pt, scale=  1.10] at (123.23,137.72) {$\vdots$};

\node[text=drawColor,anchor=base,inner sep=0pt, outer sep=0pt, scale=  1.10] at (177.64,137.72) {$\vdots$};

\node[text=drawColor,anchor=base,inner sep=0pt, outer sep=0pt, scale=  1.10] at (286.47,137.72) {$\vdots$};
\definecolor{drawColor}{RGB}{217,95,2}
\definecolor{fillColor}{RGB}{217,95,2}

\path[draw=drawColor,line width= 0.6pt,line cap=rect,fill=fillColor,fill opacity=0.10] ( 96.02, 49.16) rectangle (204.85,112.79);
\end{scope}
\begin{scope}
\path[clip] (  0.00,  0.00) rectangle (361.35,216.81);
\definecolor{drawColor}{RGB}{0,0,0}

\path[draw=drawColor,line width= 0.6pt,line join=round] ( 26.64, 30.69) --
	( 26.64,211.31);
\definecolor{fillColor}{RGB}{0,0,0}

\path[draw=drawColor,line width= 0.6pt,line join=round,fill=fillColor] ( 31.64,202.65) --
	( 26.64,211.31) --
	( 21.64,202.65) --
	cycle;
\end{scope}
\begin{scope}
\path[clip] (  0.00,  0.00) rectangle (361.35,216.81);
\definecolor{drawColor}{gray}{0.30}

\node[text=drawColor,anchor=base east,inner sep=0pt, outer sep=0pt, scale=  0.88] at ( 21.69, 56.39) {1};

\node[text=drawColor,anchor=base east,inner sep=0pt, outer sep=0pt, scale=  0.88] at ( 21.69, 97.44) {2};

\node[text=drawColor,anchor=base east,inner sep=0pt, outer sep=0pt, scale=  0.88] at ( 21.69,179.54) {$d$};
\end{scope}
\begin{scope}
\path[clip] (  0.00,  0.00) rectangle (361.35,216.81);
\definecolor{drawColor}{gray}{0.20}

\path[draw=drawColor,line width= 0.6pt,line join=round] ( 23.89, 59.42) --
	( 26.64, 59.42);

\path[draw=drawColor,line width= 0.6pt,line join=round] ( 23.89,100.47) --
	( 26.64,100.47);

\path[draw=drawColor,line width= 0.6pt,line join=round] ( 23.89,182.57) --
	( 26.64,182.57);
\end{scope}
\begin{scope}
\path[clip] (  0.00,  0.00) rectangle (361.35,216.81);
\definecolor{drawColor}{RGB}{0,0,0}

\path[draw=drawColor,line width= 0.6pt,line join=round] ( 26.64, 30.69) --
	(355.85, 30.69);
\definecolor{fillColor}{RGB}{0,0,0}

\path[draw=drawColor,line width= 0.6pt,line join=round,fill=fillColor] (347.19, 25.69) --
	(355.85, 30.69) --
	(347.19, 35.69) --
	cycle;
\end{scope}
\begin{scope}
\path[clip] (  0.00,  0.00) rectangle (361.35,216.81);
\definecolor{drawColor}{gray}{0.20}

\path[draw=drawColor,line width= 0.6pt,line join=round] ( 68.81, 27.94) --
	( 68.81, 30.69);

\path[draw=drawColor,line width= 0.6pt,line join=round] (123.23, 27.94) --
	(123.23, 30.69);

\path[draw=drawColor,line width= 0.6pt,line join=round] (177.64, 27.94) --
	(177.64, 30.69);

\path[draw=drawColor,line width= 0.6pt,line join=round] (286.47, 27.94) --
	(286.47, 30.69);
\end{scope}
\begin{scope}
\path[clip] (  0.00,  0.00) rectangle (361.35,216.81);
\definecolor{drawColor}{gray}{0.30}

\node[text=drawColor,anchor=base,inner sep=0pt, outer sep=0pt, scale=  0.88] at ( 68.81, 19.68) {1};

\node[text=drawColor,anchor=base,inner sep=0pt, outer sep=0pt, scale=  0.88] at (123.23, 19.68) {2};

\node[text=drawColor,anchor=base,inner sep=0pt, outer sep=0pt, scale=  0.88] at (177.64, 19.68) {3};

\node[text=drawColor,anchor=base,inner sep=0pt, outer sep=0pt, scale=  0.88] at (286.47, 19.68) {$n$};
\end{scope}
\begin{scope}
\path[clip] (  0.00,  0.00) rectangle (361.35,216.81);
\definecolor{drawColor}{RGB}{0,0,0}

\node[text=drawColor,anchor=base east,inner sep=0pt, outer sep=0pt, scale=  1.10] at (355.85,  7.64) {$t$};
\end{scope}
\begin{scope}
\path[clip] (  0.00,  0.00) rectangle (361.35,216.81);
\definecolor{drawColor}{RGB}{0,0,0}

\node[text=drawColor,anchor=base east,inner sep=0pt, outer sep=0pt, scale=  1.10] at ( 14.36,203.73) {$V$};
\end{scope}
\end{tikzpicture}}
    \caption{Representation of a realization of the process $\bX$, with set of vertices $V = \{1, \dots, d\}$ observed from time $1$ to $n$.
    Subscript indicates the vertex and superscript indicates the time at which the observation was taken.
    The highlighted rectangle indicates the observed slice $x_{\{1,2\}}^{(2:3)}.$}
    \label{fig:exemplo_processo}
\end{figure}

We say the processes $X^{(1)},\dotsc, X^{(n)}$ 
satisfies a mixing condition with rate $\{\psi(\ell)\}_{\ell\in\mathbb N}$   if for each $k, m$ 
and each $x^{(1:k)}$, $x^{(1:m)}$ with $\mathbb P(X^{(1:m)}=x^{(1:m)})>0$ 
we have that 
\begin{equation}\label{mixing}
\begin{split}
\bigl| \mathbb P( X^{(n:(n+k-1))}=x^{(1:k)}\, |\,& X^{(1:m)}=x^{(1:m)}) -  \mathbb P( X^{(n:(n+k-1))}=x^{(1:k)})\bigr|\\[3mm]
& \leq\; \psi(\ell)\,\mathbb P( X^{(n:(n+k-1))}=x^{(1:k)})
\end{split}
\end{equation}
for $n\geq m+\ell$.

Assume we observe a sample of size $n$ of the process, denoted by  
$\{x^{(i)}\colon i=1,\dots,n\}$.
Since the stationary distribution of the process $\pi$ is not known, we must estimate it from the data. For any $W\subset V$ and any $a_W\in A^{W}$ denote by 
\begin{equation*}
    \widehat\pi(a_W)  = \frac{N(a_W)}{n}\,,
\end{equation*}
where $N(a_W)$ denotes the number of times the configuration $a_W$ appears in the sample $x^{(1)}, \dotsc, x^{(n)}$. If $ \widehat\pi(a_W)>0$, we can also define the conditional probabilities 
\begin{equation}\label{hatpicond}
    \widehat\pi(a_W|a_{W'})  = \frac{\widehat\pi(a_{W\cup W'})}{\widehat\pi(a_{W'})}\,,
\end{equation}
for two disjoint subsets  $W,W'\subset V$ and configurations  $a_W\in A^W, a_{W'}\in A^{W'}$. 

Based on results by \citet{csiszar2002}, we can state and prove two propositions showing the rate of convergence of the empirical probabilities in a  stationary stochastic process with exponential mixing sequence.  From now on, the phrase \emph{eventually almost surely} means with probability one, for all $n$ large enough. 

\begin{proposition}\label{prop:limite_conjunto}
Assume the process $\{X^{(i)}\colon i\in \mathbb N\}$ satisfies the mixing condition \eqref{mixing}  with mixing rate $\psi(\ell) =O(1/\ell^{1+\epsilon})$, for some $\epsilon>0$. Then for any $\delta>0$, 
any $W \subset V$ and any $a_W\in A^W$ we have that 
\[
| \widehat\pi(a_W)  -  \pi(a_W)| \;<\; \sqrt{\frac{\delta\log n}{n}}
\]
eventually almost surely as $n\to\infty$. Moreover, for any disjoint sets $W,W' \subset V$ and any $a_W\in A^W$, $a_{W'}\in A^{W'}$ we have that
\begin{equation*}
    \Big\vert \widehat\pi(a_W|a_{W'}) - \pi(a_W|a_{W'}) \Big\vert < \sqrt{\frac{\delta\log n}{N(a_{W'})}},
\end{equation*}
eventually almost surely as $n \rightarrow \infty.$
\end{proposition}

The proof of Proposition~\ref{prop:limite_conjunto} is postponed to the Appendix. 

\section{The graph's estimator and its consistency}

In this paper we take a regularized pseudo maximum likelihood approach to estimate  the graph $G^*$, given a sample $x^{(1)},\dotsc, x^{(n)}$ of the stochastic process.  
Instead of estimating each neighborhood and then combining the results, as is proposed in several works, we globally estimate the graph $G^*$ by optimizing a function over the set of all simple graphs over $V$. 

Given any graph $G$ defined on the same set of vertices $V$,  the pseudo-likelihood function is defined by 
\[
L(G) \;=\;  \prod_{i=1}^n \prod_{v \in V}  \pi(x^{(i)}_v | x^{(i)}_{G(v)})\,,
\]
where $G(v)$ denotes the neighborhood of node $v$ in the graph $G$, that is $G(v)=\{u\in V\colon (u,v)\in E\}$. 
As the conditional probabilities of $\pi$ are not known, we can estimate them from the data, obtaining the  maximum pseudo-likelihood given by 
\[
\widehat L(G) \;=\;  \prod_{i=1}^n \prod_{v \in V}  \widehat\pi(x^{(i)}_v | x^{(i)}_{G(v)})\,,
\]
with $\widehat\pi(x^{(i)}_v | x^{(i)}_{G(v)})$ defined as in \eqref{hatpicond}, taking $W=\{v\}$ 
and $W'=G(v)$. 

Applying the logarithm and taking into account the number of occurrences of each configuration in the sample, we can write the log pseudo likelihood function as
\begin{equation}\label{log-pseudo}
\log \widehat L(G) \;=\;  \sum_{v \in V}\sum_{a_v, a_{G(v)}}\;   \widehat\pi(a_v|a_{G(v)})^{N(a_v,a_{G(v)})}\,,
\end{equation}
where the sum is taken over all $v\in V$ and all configurations $a_v\in A$, $a_{G(v)}\in A^{G(v)}$ such that $N(a_v,a_{G(v)})>0$. 

We then define the graph estimator by 
\begin{equation}\label{hatP}
\widehat G\;=\; \underset{G}{\arg\max} \bigl\{\, \log \widehat L(G)  - \lambda_n \sum_{v\in V} |A|^{|G(v)|} \,\bigr\}\,,
\end{equation}
with $|G(v)|$ denoting the cardinal of the set $G(v)$ and $\lambda_n$ being a non-negative decreasing sequence.  

The main result in this paper is the following consistency result for the graph estimator $\widehat G$. 

\begin{theorem}\label{consistencia}
Assume the process $\{X^{(i)}\colon i \in\mathbb Z\}$ satisfies the mixing condition \eqref{mixing} with rate $\psi(\ell) = O(1/\ell^{1+\epsilon})$ for some $\epsilon>0$. Then, taking  
$\lambda_n= c\log n$, with $c>0$,  we have that $\widehat G$ defined in \eqref{hatP}   satisfies $\widehat G=G^*$ eventually almost surely when  $\fin\to \infty$.
\end{theorem}

Before proving Theorem~\ref{consistencia}, we recall the definition of  the  K\"ullback-Leibler divergence between two probability distributions $p$ 
and $q$ over $A$. It is given by
 \begin{equation}\label{KL}
 D(p;q) = \sum_{a\in A} p(a)\log\frac{p(a)}{q(a)}
 \end{equation}
where, by convention, $p(a) \log\frac{p(a)}{q(a)}=0$ if $p(a)=0$ and $p(a) \log\frac{p(a)}{q(a)}=+\infty$ if $p(a) > q(a)=0$. An important property of the K\"ullback-Leibler divergence is that  $D(p;q) =0$ if and only if $p(a)=q(a)$ for all $a\in A$.

Denote by $G_{\max}$ the complete graph over $V$, that is $G_{\max} = (V, E_{\max})$ with 
\[
E_{\max} \;=\; \{(u,v)\in V\times V\colon u\neq v\}\,.
\]
Observe that in particular we have that $G_{\max}(v) = V\setminus\{v\}$ for all $v\in V$. 
For any $v \in V$, denote by 
\begin{equation} \label{eq:alpha}
    \alpha(v) \;=\; \min_{\substack{G\colon G^*\not\subset G}}\; \Bigg\{ \,  \sum_{a_{G_{\max}(v)}} \; \pi (a_{G_{\max}(v)}) D\big(\pi(\cdot_v|a_{G^*(v)}); \pi(\cdot_v|a_{G(v)})\big) \, \Bigg\}
\end{equation}
where $\pi(\cdot_v|a_{G^*(v)})$ denotes the probability distribution over $A$ given by  $\{\pi(a_v|a_{G^*(v)})\}_{a_v\in A}$ and similarly for $\pi(\cdot_v|a_{G(v)})$. 
By Definition~\ref{def:ne_v} we must have $\alpha(v)>0$, for a proof see \citet{leonardi_et_al2023}. 

\begin{proof}[Proof of Theorem~\ref{consistencia}]
First observe that we can decompose the event $\{\widehat G\neq G^*\}$ as the union of the two events $\{G^* \subsetneq \widehat G \} \cup  \{G^* \not\subset \widehat G\}$.  We will consider these two events separately, proving that  eventually almost surely as $n\to\infty$ neither of them can happen, implying that $\widehat G = G^*$. \\

\noindent Case (a), $P(\{G^* \subsetneq \widehat G\})=0$ e.a.s as $n\to\infty$ (non-overfitting).  To prove that this event will not happen, we will prove that for all  graphs $G\supsetneq G^*$ we have that 
\begin{equation}\label{eq:logp}
\log \widehat L(G) - \lambda_n \sum_{v\in V} |A|^{|G(v)|}  \;<\; \log \widehat L(G^*) - \lambda_n\sum_{v\in V} |A|^{|G^*(v)|}
\end{equation}
or equivalently that 
\begin{equation}\label{eq:logp2}
\log \widehat L(G)  - \log \widehat L(G^*) \;<\;  \lambda_n\,\Bigl( \sum_{v\in V} |A|^{|G(v)|} - \sum_{v\in V} |A|^{|G^*(v)|}\Bigr)
\end{equation}
eventually almost surely as $n\to\infty$, proving that $\widehat G\neq G$ for all $G\supsetneq G^*$. 
Observe that
\begin{equation*}
\log \widehat L(G)  \;=\;  \sum_{v \in V}  \sum_{a_v, a_{G(v)}} N(a_v, a_{G(v)}) \log \widehat\pi( a_v | a_{G(v)})
\end{equation*}
and similarly for $\log \widehat L(G^*)$. 
Then
\begin{equation}\label{eq3}
\begin{split}
\log \widehat L(G) - \log \widehat L(G^*) 
\; = \;&  \sum_{v \in V}\sum_{a_v,a_{G(v)}} N(a_v,a_{G(v)}) \log \widehat \pi(a_v|a_{G(v)})\\
&- \sum_{v \in V}\sum_{a_v,a_{G^*(v)}} N(a_v,a_{G^*(v)}) \log \widehat \pi(a_v|a_{G^*(v)})\,.
\end{split}
\end{equation}
Fix $v\in V$. By the definition of the maximum likelihood estimators and as $G\supset G^*$ we have that
\begin{equation}\label{eq4}
\begin{split}
  \sum_{a_v,a_{{G^*}(v)}} N(a_v,a_{{G^*}(v)}) \log \widehat \pi(a_v|a_{{G^*}(v)}) \;&\geq\; \sum_{a_v,a_{{G^*}(v)}} N(a_v,a_{{G^*}(v)})\log \pi(a_v|a_{{G^*}(v)})\\
 &= \;  \sum_{a_v,a_{{G}(v)}} N(a_v,a_{{G}(v)}) \log \pi(a_v|a_{{G^*}(v)})\,.
\end{split}
\end{equation}
Therefore, using \eqref{eq4},  the difference in \eqref{eq3}  can be upper-bounded by
\begin{equation*}\label{diff}
 \sum_{v \in V}\sum_{a_v,a_{G(v)}} N(a_v,a_{G(v)}) \log \frac{\widehat \pi(a_v|a_{G(v)})}{\pi(a_v|a_{G(v)})} \;=\;   \sum_{v \in V}\sum_{ a_{G(v)}} N(a_{G(v)})D(\widehat \pi(\cdot_v|a_{G(v)}) \,;\,\pi(\cdot_v|a_{G(v)}) )\,,
\end{equation*}
where $D$ denotes the K\"ullback-Leibler divergence, see \eqref{KL}. Therefore we have, by Lemma~\ref{KLTV}, that 
\begin{equation}\label{eqD}
\begin{split}
 \sum_{v \in V}\sum_{a_{G(v)}} N(a_{G(v)})&D(\widehat \pi(\cdot_v|a_{G(v)}) \,;\,\pi(\cdot_v|a_{G(v)}) ) \\
&\leq\;  \sum_{v \in V}\sum_{a_{G(v)}} N(a_{G(v)}) \sum_{a_v\in A} \frac{[\,\widehat \pi(a_v|a_{G(v)}) - \pi(a_v|a_{G(v)})\,]^2}{\pi(a_v|a_{G(v)})}.
\end{split}
\end{equation}
Then, by Proposition~\ref{prop:limite_condicional} with $\delta>0$ and \eqref{eqD} we have, with probability  1, for $n$ sufficiently large  that 
\begin{align*}
 \sum_{v \in V}\sum_{a_{G(v)}} N(a_{G(v)})\,&D(\widehat \pi(\cdot_v|a_{G(v)}) \,;\,\pi(\cdot_v|a_{G(v)}) ) \;\leq\;   \frac{\delta |A| \log n }{\pi_{\min}}\,\sum_{v\in V} |A|^{G(v)}\,.
\end{align*}
On the other hand we have that 
\begin{equation}
\begin{split}
\lambda_n \Bigl( \sum_{v\in V} |A|^{G(v)} -   \sum_{v\in V} |A|^{G^*(v)}  \Bigr) \;&=\; \lambda_n  \sum_{v\in V} |A|^{G(v)}\Bigl(1 - \frac{ \sum_{v\in V} |A|^{G^*(v)} }{\sum_{v\in V} |A|^{G(v)}} \Bigr)\\
&\geq \; \frac{\lambda_n}2  \sum_{v\in V} |A|^{G(v)}
\end{split}
\end{equation}
as
\[
 \frac{ \sum_{v\in V} |A|^{G^*(v)} }{\sum_{v\in V} |A|^{G(v)}} \;\leq\; \frac{1}{|A|}\;\leq\; \frac12 \,.
 \]
Then for  $\lambda_n=c\log n$ there exists $\delta>0$ such that 
\begin{equation*}
 \frac{\delta |A|}{\pi_{\min}}\;<\;  c
\end{equation*}
that is, we take $\delta < c\, \pi_{\min}/|A|$ and we have that 
\[
\max_{G\supsetneq G^*} \; \{\, \log \widehat L(G) -\lambda_n\, \sum_{v\in V}|A|^{|G(v)|}\,\}  \; <\; \log \widehat L(G^*) - \lambda_n \,\sum_{v\in V}|A|^{|G^*(v)|}  
\]
with probability 1 for $n$ sufficiently large.  
This concludes the proof of Case (a). \\[2mm]

\noindent Case (b), $P(\{G^* \not\subset \widehat G\})=0$ e.a.s as $n\to\infty$ (non-underfitting).
In order to prove this case we need to show that for any graph $G,$ such that $G\not\supset G^*$ we have that  
\begin{equation}\label{eq:dif_G_Gstar}
   \log \widehat L(G) -\lambda_n\, \sum_{v\in V} |A|^{G(v)} \; < \;   \log \widehat L(G^*) -\lambda_n\, \sum_{v\in V} |A|^{G^*(v)}
\end{equation}
eventually almost surely as $n\to\infty$.
In order to prove that \eqref{eq:dif_G_Gstar} holds, first we prove that
\begin{equation*}
     \log \widehat L(G) -\lambda_n\, \sum_{v\in V} |A|^{G(v)} \; < \;   \log \widehat L(G_{\max}) -\lambda_n\, \sum_{v\in V} |A|^{G_{\max}(v)},
\end{equation*}
with  $G_{\max}$ denoting the complete graph in $V$. 
Then, this inequality together with the arguments presented in case (a) will imply the desired result.\\
Note that we have
\begin{equation} \label{eq:dif_caso2}
\begin{split}
    \log  \widehat L(G_{\max})&\, -\lambda_n\,  \sum_{v\in V} |A|^{G_{\max}(v)} -  \log \widehat L(G) + \lambda_n\,  \sum_{v\in V} |A|^{G(v)} \\[2mm]
    &=\;  \sum_{v \in V} \sum_{a_v,a_{G_{\max}(v)}} N(a_v,a_{G_{\max}(v)}) \log \frac{\widehat \pi(a_v|a_{G_{\max}(v)})}{\widehat \pi(a_v|a_{G(v)})}\\
    &\qquad - \lambda_n\, \sum_{v\in V} \big (  |A|^{G_{\max}(v)} - |A|^{G(v)}\big) \\ 
    &= \;n  \Bigl[ \sum_{v \in V} \sum_{a_v,a_{G_{\max}(v)}} \frac{N(a_v,a_{G_{\max}(v)})}{n} \log \frac{\widehat \pi(a_v|a_{G_{\max}(v)})}{\widehat \pi(a_v|a_{G(v)})} \\
    &\qquad- \frac{\lambda_n}{n}\,\sum_{v\in V} \big (  |A|^{G_{\max}(v)} - |A|^{G(v)}\big)  \Bigr].
\end{split}
\end{equation}
One can see that for $\lambda_n=c\log n$, with $c>0$, the second term in the brackets in the last expression of  \eqref{eq:dif_caso2} vanishes when $n \rightarrow \infty$, i.e.,
\begin{equation*}
    \frac{\lambda_n}{n}\,\sum_{v\in V} \big (  |A|^{G_{\max}(v)} - |A|^{G(v)}\big) \; \xrightarrow[n \to \infty]\; 0.
\end{equation*}
Now, by adding 
\begin{equation*}
    \frac{N(a_v, a_{G_{\max}(v)})}{n}  \, \log \frac{\pi(a_v \vert a_{G(v)})}{\pi(a_v \vert a_{G(v)})} \;=\;0
\end{equation*}
into the first term of the sum in \eqref{eq:dif_caso2}, we can write it as
\begin{align} \label{eq:decomposicao_pt3}
    \sum_{v \in V}  \sum_{a_v,a_{G_{\max}(v)}} & \frac{N(a_v,a_{G_{\max}(v)})}{n} \log \frac{\widehat \pi(a_v|a_{G_{\max}(v)})}{\widehat \pi(a_v|a_{G(v)})} \nonumber \\
    =&\; \sum_{v \in V}  \sum_{a_v,a_{G_{\max}(v)}} \frac{N(a_v,a_{G_{\max}(v)})}{n} \bigg(\log \frac{\widehat \pi(a_v|a_{G_{\max}(v)})}{\widehat \pi(a_v|a_{G(v)})} + \log \frac{\pi(a_v \vert a_{G(v)})}{\pi(a_v \vert a_{G(v)})}\bigg) \nonumber \\
    =& \;\sum_{v \in V} \sum_{a_v,a_{G_{\max}(v)}} \frac{N(a_v,a_{G_{\max}(v)})}{n} \bigg( \log \frac{\widehat \pi(a_v|a_{G_{\max}(v)})}{ \pi(a_v|a_{G(v)})} - \log \frac{ \widehat \pi(a_v|a_{G(v)})}{ \pi(a_v|a_{G(v)})} \bigg) \nonumber \\
    =& \;\underbrace{\sum_{v \in V} \sum_{a_v,a_{G_{\max}(v)}} \frac{N(a_v,a_{G_{\max}(v)})}{n} \log \frac{\widehat \pi(a_v|a_{G_{\max}(v)})}{ \pi(a_v|a_{G(v)})}}_{(1)} \nonumber \\
    & \hspace{50pt}  - \: \underbrace{\sum_{v \in V} \sum_{a_v,a_{G_{\max}(v)}} \frac{N(a_v,a_{G_{\max}(v)})}{n} \log \frac{\widehat \pi(a_v|a_{G(v)})}{ \pi(a_v|a_{G(v)})}}_{(2)}.
\end{align}
As highlighted in \eqref{eq:decomposicao_pt3}, we analyse this expression in two parts.
The second term in the right-hand side of \eqref{eq:decomposicao_pt3} can be written as 
\begin{align*}
    \sum_{v \in V} \sum_{a_v,a_{G(v)}} &\frac{N(a_{G(v)})}{n}\, \widehat\pi(a_v|a_{G(v)}) \log \frac{\widehat \pi(a_v|a_{G(v)})}{\pi(a_v|a_{G(v)})} \\
    =&\; \sum_{v \in V} \sum_{a_{G(v)}} \frac{N(a_{G(v)})}{n}\; \sum_{a_v}\widehat\pi(a_v|a_{G(v)}) \log \frac{\widehat \pi(a_v|a_{G(v)})} {\pi(a_v|a_{G(v)})} \\
    =& \;\sum_{v \in V} \sum_{a_{G(v)}} \frac{N(a_{G(v)})}{n}\; D\big(\widehat \pi(\cdot_v|a_{G(v)}); \pi(\cdot_v|a_{G(v)})\big).
\end{align*}
Observe that, by Lemma~\ref{KLTV} and Proposition~\ref{prop:limite_condicional}, for $\delta>0,$ we have 
\begin{equation*}\label{eq:D2}
\begin{split}
    \sum_{v \in V}\sum_{a_{G(v)}} \frac{N(a_{G(v)})}{n}&\;D(\widehat \pi(\cdot_v|a_{G(v)}) \,;\,\pi(\cdot_v|a_{G(v)}) ) \\
    &\leq\;  \sum_{v \in V}\sum_{a_{G(v)}}  \frac{N(a_{G(v)})}{n} \;\sum_{a_v\in A} \frac{[\,\widehat \pi(a_v|a_{G(v)}) - \pi(a_v|a_{G(v)})\,]^2}{\pi(a_v|a_{G(v)})} \\
    & \leq \sum_{v \in V}\sum_{a_{G(v)}} \frac{N(a_{G(v)})}{n} \;\sum_{a_v\in A} \frac{\delta \log n }{N(a_{G(v)})\pi(a_v|a_{G(v)})} \\
    & \leq |V| |A|^{|V|} \,\frac{\delta}{\pi_{\min}}\frac{\log n}{n} \rightarrow 0
\end{split}
\end{equation*}
as $n \to\infty.$ 
On the other hand, since $\widehat \pi(a_v|a_{G_{\max}(v)})$ are the maximum likelihood estimators of 
$\pi(a_v|a_{G_{\max}(v)})$ and $G^* \subseteq G_{\max}$, the first term in the right-hand side of \eqref{eq:decomposicao_pt3} can be lower-bounded  by 
\begin{equation}\label{eq:limite_inferior}
\begin{split}
    \sum_{v \in V} \sum_{a_v,a_{G_{\max}(v)}} & \frac{N(a_v,a_{G_{\max}(v)})}{n} \log \frac{\widehat \pi(a_v|a_{G_{\max}(v)})}{ \pi(a_v|a_{G(v)})}  \\
    & \hspace{20pt} \geq \: \sum_{v \in V} \sum_{a_v,a_{G_{\max}(v)}} \frac{N(a_v,a_{G_{\max}(v)})}{n} \log \frac{ \pi(a_v|a_{G_{\max}(v)})}{ \pi(a_v|a_{G(v)})}  \\
     & \hspace{20pt} =  \: \sum_{v \in V} \sum_{a_v,a_{G_{\max}(v)}} \frac{N(a_v,a_{G_{\max}(v)})}{n} \log \frac{ \pi(a_v|a_{G^*(v)})}{ \pi(a_v|a_{G(v)})} 
\end{split}
\end{equation}
By Proposition~\ref{prop:limite_conjunto}
\begin{equation*}
    \frac{N(a_v,a_{G_{\max}(v)})}{n}\; =\; \widehat \pi (a_v,a_{G_{\max}(v)}) \;>\; \pi (a_v,a_{G_{\max}(v)}) - \sqrt{\frac{\delta \log n}{n}},
\end{equation*}
eventually almost surely as $n \rightarrow \infty$.
Then, one can see that \eqref{eq:limite_inferior} can be lower-bounded by
\begin{equation}\label{eq:lowerbound}
\begin{split}
    \sum_{v \in V} \sum_{a_v,a_{G_{\max}(v)}} &  \Bigg[\pi (a_v,a_{G_{\max}(v)}) - \sqrt{\frac{\delta \log n}{n}} \Bigg]\log \frac{ \pi(a_v|a_{G^*(v)})}{ \pi(a_v|a_{G(v)})} \\
    & = \sum_{v \in V}\sum_{a_{G_{\max}(v)}} \pi(a_{G_{\max}(v)}) \sum_{a_v} \pi (a_v|a_{G_{\max}(v)}) \log \frac{ \pi(a_v|a_{G^*(v)})}{ \pi(a_v|a_{G(v)})} \\
    & \hspace{100pt} - \sqrt{\frac{\delta \log n}{n}} \sum_{v \in V}\sum_{a_v,a_{G_{\max}(v)}}\log \frac{ \pi(a_v|a_{G^*(v)})}{ \pi(a_v|a_{G(v)})}\\
    & = \sum_{v \in V}\sum_{a_{G_{\max}(v)}} \pi(a_{G_{\max}(v)}) D\big(\pi(\cdot_v|a_{G^*(v)}); \pi(\cdot_v|a_{G(v)})\big) \\
    & \hspace{100pt} + \sqrt{\frac{\delta \log n}{n}} |V| |A|^{|V|} \log \pi_{\min}\\
    & \geq \frac{1}{2}\sum_{v \in V}\alpha(v).
\end{split}
\end{equation}
eventually almost surely as $n \rightarrow \infty.$
Therefore, since $\sum_{v \in V}\alpha(v) > 0$, we have from \eqref{eq:lowerbound} that
\begin{equation*} 
   \log \widehat L(G) -\lambda_n\, \sum_{v\in V} |A|^{G(v)} \;< \;   \log \widehat L(G_{\max}) -\lambda_n\,\sum_{v\in V} |A|^{G_{\max}(v)},
\end{equation*}
eventually almost surely as $n \rightarrow \infty.$
Now, since $G^* \subseteq G_{\max}$, by Case (a) we have that
\begin{equation*}
  \log \widehat L(G_{\max}) -\lambda_n\,\sum_{v\in V} |A|^{G_{\max}(v)}\; \leq\; \log \widehat L(G^*) -\lambda_n\,
  \sum_{v\in V}|A|^{G^*(v)}
\end{equation*}
eventually almost surely as $n \rightarrow \infty,$ and this concludes the proof for Case (b).
Thus, combining the two cases leads to
\begin{equation*}
    \max_{G\colon G\neq G^*} \log \widehat L(G) -\lambda_n\,\sum_{v\in V}|A|^{G(v)}\;<\; \log \widehat L(G^*) -\lambda_n\,\sum_{v\in V}  |A|^{G^*(v)}
\end{equation*}
eventually almost surely as $n \rightarrow \infty$ and this concludes the proof of Theorem~\ref{consistencia}.
\end{proof}

\section*{Discussion}

In this paper, we introduced a model selection approach to estimate the underlying graph of conditional dependencies in a multivariate stochastic process. 
Our method relies on a penalized pseudo-likelihood and employs a global estimation approach.
We have established the almost sure convergence of this estimator to the true underlying graph, specifically in the context of finite graphical models, provided a certain mixing condition is satisfied.
While the case of independent and identically distributed processes has been extensively explored in the literature, this assumption often proves too restrictive for real-world applications where independence does not hold. 

Our approach distinguishes itself by considering the estimation of the entire graph at once, diverging from the usual practice found in the literature, which typically estimates individual neighborhoods for each vertex and subsequently combines them to form the graph. 

In practical terms, the computation of the proposed estimator poses a significant computational challenge, as it involves searching through all potential graph configurations.
To address this, iterative algorithms such as simulated annealing or stepwise greedy algorithms could be employed to facilitate an efficient approximation. Moreover, the definition of the estimator is based on a penalization constant that must be stated before the analysis. The choice of this constant is a challenging problem in regularized approaches and could be addressed with methods like cross-validation. 

Looking forward, there are several promising directions for extending this work. One such direction involves adapting the theoretical framework to accommodate continuous multivariate stochastic processes, thereby broadening the range of potential applications of our methodology. Another line of research is the generalization to infinite vertex sets and unbounded estimators, where the size of the estimated graph is allowed to grow with the sample size. 

\section*{Appendix}
\begin{lemma} \label{lemma:submixing}
If the process $\{X^{(i)}\colon i\in \mathbb N\}$ satisfies the mixing condition \eqref{mixing} with rate $\{\psi(\ell)\}$ then the sub-process $\{X^{(i)}_W\colon i\in \mathbb N\}$ with $W\subset V$ also satisfies the mixing condition with rate $\{\psi(\ell)\}$.
\end{lemma}

\begin{proof}
Observe that  for any $x_{W}^{(1:k)}, x_W^{(1:m)} \in A^W$  we have that 
\begin{align*}
\begin{split}
\bigl|& \prob( X^{(n:(n+k-1))}_W=x^{(1:k)}_W\, |\, X^{(1:m)}_W=x^{(1:m)}_W) -  \prob( X^{(n:(n+k-1))}_W=x^{(1:k)}_W)\bigr| \\[2mm]
&= \; \Bigl| \;\sum_{x_{W^c}^{(1:k)}}   \prob( X^{(n:(n+k-1))}_{W\cup W^c}=x^{(1:k)}_{W\cup W^c}\,|\, X^{(1:m)}_W=x^{(1:m)}_W)\\
& \hspace{100pt} - \sum_{x_{W^c}^{(1:k)}}   \prob( X^{(n:(n+k-1))}_{W\cup W^c}=x^{(1:k)}_{W\cup W^c})\Bigr|\\
&\leq \sum_{x_{W^c}^{(1:k)}}  \psi(n-m) \prob(X^{(n:(n+k-1))}_{W\cup W^c}=x^{(1:k)}_{W\cup W^c})\\
&\leq \psi(n-m) \prob(X_W^{(n:(n+k-1))} = x_W^{(1:k)})\,.
\end{split}
\end{align*}
Then the process $\{X^{(i)}_W\colon i\in \mathbb N\}$ with $W\subset \{1,\dots,d\}$ is mixing with rate $\psi(\ell)$.
\end{proof}

\begin{lemma} \label{lemma:mixing_function}
Let the process $\{X^{(i)}\colon i\in \mathbb N\}$ satisfy the mixing condition \eqref{mixing} with rate $\{\psi(\ell)\}_{\ell\in\mathbb N}$ and let 
$f: A^d \rightarrow \mathbb{R}$ be a function.
Then, the process $\big\{f\big(X^{(i)}\big)\colon i\in \mathbb N \big\}$ is also mixing with rate $\{\psi(\ell)\}_{\ell\in \mathbb N}$.
\end{lemma}

\begin{proof}
Denote by $Y^{(i)} = f\big(X^{(i)}\big)$ and let $\prob_Y$ denote the distribution of the process  $\big\{Y^{(i)}\colon i\in \mathbb N \big\}$.
Now, for $i<j$, define the set
\begin{equation*}
  C(y^{(i:j)}) = \big\{\big(x^{(i)}, \dots, x^{(j)}\big) \in A^{d\times {(j-i+1)}}: Y^{(i:j)}=y^{(i:j)}\big\},
\end{equation*}
which denotes all configurations $\big(x^{(i)}, \dots, x^{(j)}\big)$ such that $\{Y^{(i:j)}=y^{(i:j)}\}$ holds.
Then
\begin{equation} \label{eq:proby1m_cond}
    \begin{split}
    \prob_Y \big(&Y^{(n:(n+k-1))} = y^{(1:k)}\, \big\vert\,  Y^{(1:m)} = y^{(1:m)}\big) \\[2mm]
    & = \sum_{x^{(1:k)} \in C(y^{(1:k)})} \prob\big(X^{(n:(n+k-1))} = x^{(1:k)} \,\big\vert  \, \cup_{x^{(1:m)} \in C(y^{(1:m)})} \big\{X^{(1:m)} = x^{(1:m)} \big\} \big)\\
    & = \sum_{x^{(1:k)} \in C(y^{(1:k)})} \frac{\sum_{x^{(1:m)} \in C(y^{(1:m)})}  \prob\big(X^{(n:(n+k-1))} = x^{(1:k)} \,, X^{(1:m)} = x^{(1:m)}\big)}{\sum_{x^{(1:m)} \in C(y^{(1:m)})}  \prob\big( X^{(1:m)} = x^{(1:m)}\big)}
    \end{split}
\end{equation}
and similarly
\begin{equation*} \label{eq:proby1m_cond2}
    \prob_Y \big(Y^{(n:(n+k-1))} = y^{(1:k)}\big) \;=\;
    \sum_{x^{(1:k)} \in C(y^{(1:k)})} \prob\big(X^{(n:(n+k-1))} = x^{(1:k)} \big)\,.
\end{equation*}
Observe that by the mixing property \eqref{mixing} we obtain, for each $x^{(1:m)} \in C(y^{(1:m)})$, that 
\begin{equation}\label{basicineq20}
\begin{split}
[1-\psi(n-m)] &\,\P( X^{(n:(n+k-1))}=x^{(1:k)}) \;\leq\; \\
&\frac{ \prob\big(X^{(n:(n+k-1))} = x^{(1:k)} \,, X^{(1:m)} = x^{(1:m)}\big)}{  \prob\big( X^{(1:m)} = x^{(1:m)}\big)} \\
&\quad\leq\; [1+\psi(n-m)] \,\P( X^{(n:(n+k-1))}=x^{(1:k)})\,.
\end{split}
\end{equation}
Then, substituting the inequalities \eqref{basicineq20} in \eqref{eq:proby1m_cond} we obtain that 
\begin{equation*}
\begin{split}
 \big\vert \;& \prob_Y\big(Y^{(n:(n+k-1))} = y^{(1:k)} \;\big\vert\;  Y^{(1:m)} = y^{(1:m)}\big) - \prob_Y\big(Y^{(n:(n+k-1))} = y^{(1:k)}\big) \; \big\vert\\[3mm]
     &\leq \sum_{x ^{(1:k)} \in C(y^{(1:k)})} \Bigl|
      \frac{\sum_{x^{(1:m)} \in C(y^{(1:m)})}  \prob\big(X^{(n:(n+k-1))} = x^{(1:k)} \,, X^{(1:m)} = x^{(1:m)}\big)}{\sum_{x^{(1:m)} \in C(y^{(1:m)})}  \prob\big( X^{(1:m)} = x^{(1:m)}\big)} \\
      &\qquad\qquad\qquad\qquad       -   \prob\big(X^{(n:(n+k-1))} = x^{(1:k)} \big)\Bigr|
     \\    
        &\leq \sum_{x ^{(1:k)} \in C(y^{(1:k)})} \psi(n-m)  \prob\big(X^{(n:(n+k-1))} = x^{(1:k)}\big) \\
       & = \psi(n-m)\, \prob_Y \big (Y^{(n:(n+k-1))}=y^{(1:k)} \big)
\end{split}
\end{equation*}
and the process $\{Y^{(i)}, i \in \mathbb N\}$ is mixing with rate $\{\psi(\ell)\}$. 
\end{proof}

The following result states a Law of the Iterated Logarithm for stochastic processes satisfying the mixing condition \eqref{mixing}, and the proof is based on the classical result by \cite{oodaira1971}.
This result is essential to prove the rate of convergence of the empirical probabilities in Proposistion~\ref{prop:limite_conjunto}. 

\begin{theorem}\label{LIL_oodaira}
Let $Y^{(i)} = f(X^{(i)})$, with $\{X^{(i)}\colon i\in \mathbb Z\}$ satisfying the mixing condition 
\eqref{mixing} with rate $\psi(
\ell) =  O(1/\ell^{1+\epsilon})$, for some $\epsilon > 0$. 
Define $Z_n = \sum_{i=1}^{n}Y^{(i)}$. Then 
\begin{equation}\label{Zn}
\vert Z_n \vert < (1+\epsilon)\sqrt{2\sigma^2 n \log \log n}
\end{equation}
eventually almost surely as $n\to\infty$, where
\begin{equation}
\sigma^2 = E\big[\big(Y^{(1)}\big)^2\big] + 2 \sum_{j=2}^{n}E(Y^{(1)}Y^{(j)})\,.
\end{equation}
\end{theorem}

\begin{proof}
The proof follows by \cite[Theorem~3]{oodaira1971}, that states that \eqref{Zn} holds under the pair of hypotheses 
\begin{enumerate}
    \item $E|Y^{(i)}|^{2+\delta} < \infty$ for some $\delta > 0;$
    \item $\psi(\ell) = O(1/\ell^{1+\epsilon})$ for some $\epsilon > 1/(1+\delta)$\,,
\end{enumerate}
as $\{Y^{(i)}\colon i\in\mathbb Z\}$ satisfies the mixing property with the same rate as $\{X^{(i)}\colon i\in\mathbb Z\}$. Moreover, as we are considering stochastic processes $\{X^{(i)}\colon i\in \mathbb Z\}$ defined over a  finite alphabet $A$ and with  $V$ finite, then we have that $E|Y^{(i)}|^{2+\delta} < \infty$ for all $\delta>0$. 
This implies the stated result.  
\end{proof}

\begin{remark}\label{remark_sigma}
A simple calculation shows that if  $\mathbb E(Y^{(1)})=0$ then $\sigma^2 = E\big[\big(Y^{(1)}\big)^2\big] = \text{Var}(Y^{(1)})$. 
\end{remark}

\begin{proof}[Proof of Proposition~\ref{prop:limite_conjunto}]
Fix $a_W \in A^{W}$ and define 
\begin{equation} \label{eq:def_Y2}
Y^{(i)}  = \mathds{1}\{X_W^{(i)} = a_W\} - \pi(a_W),
\end{equation}
for $i = 1,2, \dots, n$.
As $\{X^{(i)}\colon i \in \mathbb N\}$ is mixing with rate $\{\psi(\ell)\}_{\ell\in\mathbb N}$, then by Lemmas~\ref{lemma:submixing} and \ref{lemma:mixing_function} the process $\{Y^{(i)}\colon i \in \mathbb N\}$ is also mixing with the same rate. Also note that $ \mathbb{E}(Y^{(i)}) =0$ 
and $\text{Var}(Y^{(i)})\leq \frac14$.
Now define the partial sum $Z_n = \sum_{i=1}^{n} Y^{(i)}$. 
By Theorem~\ref{LIL_oodaira}, for $\epsilon > 0$ we have 
that
\begin{equation*}
    \vert Z_n \vert < (1+\epsilon) \sqrt{2\sigma^2 n\log\log n},
\end{equation*}
eventually almost surely as $n \to \infty$, with $\sigma^2 \leq \frac14$ (see Remark~\ref{remark_sigma}).
Since, by definition, $Z_n = N(a_W) - n\pi(a_W),$ we obtain that 
\begin{equation*}
  \big\vert n^{-1} Z_n \big\vert \;=\;    \big\vert \widehat\pi(a_W) - \pi(a_W) \big\vert \;<\; \sqrt{\frac{\log\log n}{n}},
\end{equation*}
eventually almost surely as $n \rightarrow \infty.$
Now, for any $\delta > 0$ we have that
\begin{equation*}
     \log\log n < \delta \log n
\end{equation*}
for all $n$ sufficiently large. 
Therefore,
\begin{equation*}
    \big\vert \widehat\pi(a_W) - \pi(a_W) \big\vert < \sqrt{\frac{\delta\log n}{n}}
\end{equation*}
eventually almost surely as $n \to\infty$, and this finishes the proof of the first part of Proposition~\ref{prop:limite_conjunto}. 
To prove the second part, fix the two disjoint subsets $W,W' \subset V$ and the configurations $a_W \in A^{W}, a_{W'}\in A^{W'}$. Define the process $\{Y^{(i)}\colon i\in\mathbb Z\}$ by 
\begin{equation} \label{eq:def_Y}
   Y^{(i)} = \mathds{1}\{X_W^{(i)} = a_W, X_{W'}^{(i)} = a_{W'}\} - \pi(a_W \vert a_{W'})\mathds{1}\{X_W^{(i)} = a_{W'}\}
\end{equation}
and let 
\begin{equation} \label{eq:def_zn}
    Z_n = \sum_{i=1}^{n} Y^{(i)}.
\end{equation}
Analogously to the first part, 
since $\{X^{(i)}: i \in \mathbb N\}$ is mixing with rate $\psi(\ell)=O(1/\ell^{1+\epsilon})$,
$\mathbb{E}(Y^{(i)}) =0$ and $\sigma^2=\text{Var}(Y^{(i)})\leq \frac14 \pi(a_W)$, 
for any $\epsilon > 0$, by Theorem~\ref{LIL_oodaira} we obtain that 
\begin{equation}\label{eq:LIL_aplicado}
    \vert Z_n \vert < (1+\epsilon) \sqrt{2\sigma^2 n\log\log n},
\end{equation}
eventually almost surely as $n \to\infty$. 
Note that $Z_n$ defined in \eqref{eq:def_zn} can be written as
\begin{equation*}
    Z_n = N(a_{W\cup W'}) - \pi(a_W \vert a_{W'}) N(a_{W'}).
\end{equation*}
If we divide $Z_n$ by $N(a_{W'})$ and take $\epsilon=\sqrt{2}-1$, by \eqref{eq:LIL_aplicado} we get that
\begin{equation*}
    \big\vert \widehat\pi(a_W \vert a_{W'}) - \pi(a_W \vert a_{W'}) \big\vert \;<\; \sqrt{\frac{\pi(a_{W'}) n \log\log n}{N(a_{W'})^2}}
\end{equation*}
eventually almost surely as $n \rightarrow \infty.$
By Proposition~\ref{prop:limite_conjunto}, for any $\alpha>0$ we have that 
\[
N(a_{W'}) \; > \; n\pi(a_{W'}) - \sqrt{\delta n \log n} \;>\; (1-\alpha) n\pi(a_{W'})
\]
eventually almost surely as $n\to\infty$. Then we obtain that 
\begin{equation*}
    \big\vert \widehat\pi(a_W \vert a_{W'}) - \pi(a_W \vert a_{W'}) \big\vert \;<\; \sqrt{\frac{\log\log n}{(1-\alpha)N(a_{W'})}}\,.
\end{equation*}
As before, for any $\delta>0$ we have that 
\[
\frac{\log\log n}{1-\alpha} \;< \delta \log n
\]
for sufficiently large $n$, and therefore for all $\delta>0$
\begin{equation*}
    \big\vert \widehat\pi(a_W \vert a_{W'}) - \pi(a_W \vert a_{W'}) \big\vert \;<\; \sqrt{\frac{\delta\log n}{N(a_{W'})}}
\end{equation*}
eventually almost surely as $n\to\infty$. 
\end{proof}

The following basic result about the K\"ullback-Leibler divergence corresponds to
 \cite[Lemma~6.3]{csiszar2006b}.  We omit its proof here. 

\begin{lemma}\label{KLTV}
For any $P$ and $Q$ we have
\[
D(P;Q)  \;\leq\; \sum_{a\in A\colon Q(a)>0} \frac{[P(a)-Q(a)]^2}{Q(a)}\,.
\]
\end{lemma}

\bibliographystyle{apalike}
\bibliography{references.bib}
\end{document}